\documentclass[a4paper,10pt]{article}

\usepackage[utf8]{inputenc}
\usepackage[english]{babel}
\usepackage[T1]{fontenc}
\usepackage{amsmath} 
\usepackage{amsfonts}
\usepackage{amssymb}
\usepackage{amsthm}
\usepackage{eufrak}
\usepackage{bbm}
\usepackage[disable]{todonotes}
\usepackage{mathtools}
\usepackage{hyperref}

\newcommand{\R}{{\mathbb{R}}}

\newcommand{\N}{{\mathbb{N}}}

\newcommand{\prog}{{\mathrm{Prog}}}

\newcommand{\off}{{\mathcal{O}}}

\renewcommand{\epsilon}{\varepsilon}
\renewcommand\O[1]{\mathcal O_{#1}}

\newcommand{\xfat}{{X_{\overline{\jmath}}}}
\newcommand{\energy}{{\mathcal{E}}}

\theoremstyle{plain}
\newtheorem{thm}{Theorem}[section]
\newtheorem{lem}[thm]{Lemma}
\newtheorem{prop}[thm]{Proposition} 

\theoremstyle{definition} 
\newtheorem{defi}{Definition}

\newtheorem{rmk}{Remark}

\theoremstyle{remark}

\title{Uniqueness for an inviscid stochastic dyadic model on a tree}
\author{Luigi Amedeo Bianchi}
\date{}

\begin{document}
\maketitle
\listoftodos

\begin{abstract}
   In this paper we prove that the lack of uniqueness for solutions of
   the tree dyadic model of turbulence is overcome with the
   introduction of a suitable noise. The uniqueness is a weak
   probabilistic uniqueness for all $l^2$-initial conditions and is
   proven using a technique relying on the properties of the
   $q$-matrix associated to a continuous time Markov chain.
\end{abstract}

\section{Introduction}
\todo[inline,caption={Stress similarities with [5]}]{The author should stress out in the introduction that
  the results have already been proven in [5] putting in evidence the
  new formulation of Section6. He has to clarify what has been done in
the previous literature and would be the novelty in his paper.}
The deterministic dyadic model on a tree
\begin{equation}
\begin{cases}
dX_{j}=(c_{j}X_{\overline{\jmath}}^{2}-\sum_{k\in\off_{j}}%
c_{k}X_{j}X_{k}) dt\\
X_0(t)=0,
\end{cases}
\label{tree-like_dyadic_det}
\end{equation}
was introduced as a wavelet description of Euler equations
in~\cite{MR2095627} and studied  in~\cite{BarBiaFlaMor12} as a model
for energy cascade in turbulence. It can be seen as a generalization
with more structure of the so called dyadic model of turbulence,
studied in~\cite{MR2746670}.
As we show in section~\ref{sec:nonuniqdet}
this deterministic model~\eqref{tree-like_dyadic_det} does not have
uniqueness in $l^2$. The aim of this paper is to prove that we can
restore uniqueness with the introduction of a suitable random noise:
\begin{equation}
dX_{j}=(c_{j}X_{\overline{\jmath}}^{2}-\sum_{k\in \O j}%
c_{k}X_{j}X_{k}) dt +\sigma  c_j X_{\overline \jmath}\circ dW_j -
\sigma \sum_{k\in \O j}c_kX_k\circ dW_k,\label{tree-like_dyadic_sto}
\end{equation}
 with $\left(W_j\right)_{j\in J}$ a sequence of independent Brownian
 motions. Let's also assume deterministic initial conditions
for~\eqref{tree-like_dyadic_sto}:
$X\left(0\right)=x=\left(x_j\right)_{j\in J} \in l^2$.
 The main result of this paper is the weak uniqueness of solution
 for~\eqref{tree-like_dyadic_sto}, proven in
 theorem~\ref{thm:weakuniq}. 

This paper can be seen as a generalization to the dyadic tree model of the
results proven for the classic dyadic model
in~\cite{BarFlaMor2010PAMS}, but the proof of uniqueness given here
relies of a new, different approach (see also~\cite{BarMor2012}) based on a
general abstract property instead of a trick (see
Section~\ref{sec:assmarkcha}). The $q-$matrix we rely 
on is closely related to an infinitesimal generator, so the technique
is valid for a larger class of models.

\todo[inline,caption={Sets $J$ and $\O j$}]{I think that the introduction of the countable sets $J$
  and $\O j$ is not a good idea. I don't see the difference between
  the two. The author should clarify or change the notation.}
\todo[inline,color=green,nolist]{As stated below one is a subset of the
  other: $J$ is the countable set of all nodes of the tree, while $\O
  j$ is the finite subset of all the children of node $j$}
The set $J$ is a countable set and its elements are called nodes. We
assume for the nodes a tree-like 
structure, where given $j\in J$, $\overline \jmath$ is the (unique) father of
the node $j$, and $\O j\subset J$ is the finite set of offsprings of $j$. 
In $J$ we identify a special node, called root and denoted by $0$. It
has no parent inside $J$, but with slight notation abuse we 
will nevertheless use the symbol $\bar0$ when needed. 

We see the nodes as eddies of different sizes, that split and transfer
their kinetic energies to smaller eddies along the tree. To
formalize this idea we consider the eddies as belonging to discrete
levels, called \emph{generations}, defined as follows. For all
$j\in J$ we define the generation number $|j|\in \N$ such that
$|0|=0$ and $|k|=|j|+1$ for all $k\in \O j$.

To every eddy $j\in J$ we associate an \emph{intensity} $X_j(t)$ at
time $t$, such that $X_j^2(t)$ is the kinetic energy of the eddy $j$
at time $t$. The relations among intensities are those given
in~\eqref{tree-like_dyadic_det} for the deterministic model
and~\eqref{tree-like_dyadic_sto} for the perturbed stochastic model. 
The coefficients $c_j$ are positive real numbers that represent the
speed of the energy flow on the tree.

The idea of a stochastic perturbation of a deterministic model is well
established in the literature, see~\cite{BarFlaMor11} for the
classical dyadic model, but also~\cite{MR2841920}, \cite{MR2305383}
for different models. This stochastic dyadic model falls in the family
of shell stochastic models. Deterministic shell models have been
studied extensively in~\cite{MR2251486} while stochastic versions have
been investigated for example
in~\cite{MR2570011} and \cite{MR2861257}.
 

When dealing with uniqueness of solutions in stochastic shell models,
the inviscid case we study is more difficult than the 
viscous one, since the more regular the space is, the simpler the
proof and the operator associated to the viscous system regularizes,
see for example~\cite{MR2281460} about GOY models, where the results
are proven only in the viscous case.

In~\eqref{tree-like_dyadic_sto} the parameter $\sigma\neq 0$ is
inserted just to stress the open problem of the zero noise limit, for
$\sigma \to 0$. This has provided an interesting selection result for
simple examples of linear transport equations (see~\cite{MR2543977}), but it
is nontrivial in our nonlinear setting, due to the singularity that
arises with the Girsanov transform, for example in~\eqref{eq:defM}. 

It is worth noting that the form of the noise is unexpected: one could
think that the stochastic part would mirror the deterministic one,
which is not the case here, since there is a $j$-indexed Brownian
component where we'd expect a $\overline\jmath$ one, and there is a
$k$-indexed one instead of a $j$ one.

One could argue that this is not the only possible choice for the
random perturbation. On one hand we chose a
multiplicative noise, instead of an additive one, but this is due to
technical reasons (see~\cite{MR2796837}).
 On the other hand, there are other possible choices, for example the
 Brownian motion could depend on the father and 
not on the node itself, so that brothers would share the same Brownian
motion. But the choice we made is dictated by the fact that we'd like to
have a formal conservation of the energy, as we have in the
deterministic case (see~\cite{BarBiaFlaMor12}).
If we  use It\^o formula to calculate
\begin{align*}
  \frac12 dX_j^2 &= X_j\circ dX_j \\
&= (c_jX_{\bar\jmath}^2X_j-
  \sum_{k\in\off_j}c_kX_j^2X_k)dt + \sigma c_jX_{\bar\jmath}X_j\circ
  dW_j - \sigma
  \sum_{k\in\off_j}c_kX_jX_k\circ dW_k,
\end{align*}
 we can sum \emph{formally} on the first $n+1$ generations, taking
 $X_0(t)=0$:
\begin{align*}
  \label{eq:formal_energy}
  \sum_{|j|=0}^n \frac12 dX_j^2 &= -\sum_{|j|=n}\sum_{k\in\O
    j}[c_kX_j^2X_k dt + \sigma c_kX_jX_k\circ dW_k] \\
  & =  -\sum_{|j|=n}\sum_{k\in\O
    j}c_kX_jX_k(X_jdt + \sigma\circ dW_k),
\end{align*}
 since the series is telescoping in both the drift and the diffusion
 parts independently. That means we have P-a.s. the formal conservation of energy, if we define
 the energy as
\begin{equation*}
  \label{eq:energy}
  \energy_n(t)=\sum_{|j|\leq n}X_j^2(t)\qquad
  \energy(t)=\sum_{j\in J}X_j^2(t)=\lim_{n\to \infty}\energy_n(t).
\end{equation*}

\section{Non-uniqueness in the deterministic case}\label{sec:nonuniqdet}
In~\cite{MR2746670} it has been proven that there exists examples of
non uniqueness of $l^2$ solutions for the dyadic model if we consider
solutions of the form $Y_n(t)=\dfrac{a_n}{t-t_0}$, called self-similar
solutions, with $(a_n)_n\in l^2$.
Thanks to the lifting result (Proposition 4.2 in~\cite{BarBiaFlaMor12}) that is enough
to obtain two different solutions of the dyadic tree model, with the same
initial conditions. 

\todo[inline,caption={reference introducing self similar solutions}]{The author should point out to reference [5] istead of
  [4] where these self-similar solutions have been introduced
  first. Moreover, [5] was dealing exactly with the same inviscid
  equation (2).}
\todo[inline, color=green,nolist]{It's good to point out that self similar
  solutions have been introduced in [5] for the dyadic model, but in
  [4] for the dyadic tree model. Moreover it isn't completely correct
  to say that the equation is the same, since the noise is different
  and there are many more interactions in the tree model.}
Following the same idea of self-similar solutions, introduced
in~\cite{MR2746670} and~\cite{BarFlaMor2010PAMS} for the classic
dyadic model and in section 5.1 in~\cite{BarBiaFlaMor12}  for the tree
dyadic model, we can construct a direct counterexample to uniqueness
of solutions. 
In order to do this we need an existence result stronger than the one
proven in~\cite{BarBiaFlaMor12}.

\begin{thm}\label{thm:exsoll2}
  For every $x\in l^2$ there exists at least one finite energy
  solution of~\eqref{tree-like_dyadic_det}, with initial conditions
  $X(0)=x$ and such that
\[
  \sum_{j\in J}X_j^2(t)\leq \sum_{j\in J}X_j^2(s)\qquad \forall\, 0\leq
  s\leq t.
\]
\end{thm}

The proof of this theorem is classical, via Galerkin
approximations, and follows that of theorem 3.3 in~\cite{BarBiaFlaMor12}. 

Now we recall the time reversing technique. We may consider the
system~\eqref{tree-like_dyadic_det} for $t\leq 0$: given a solution
$X(t)$ of this system for $t\geq 0$, we can define
$\widehat{X}(t)=-X(-t)$, which is a solution for $t\leq 0$, since

\begin{align*}
  \frac{d}{dt}\widehat{X}_j(t)=\frac{d}{dt}X_j(-t)&=c_jX_{\overline{\jmath}}^2(-t)
  - \sum_{k\in \O j}c_kX_j(-t)X_k(-t)\\
& =c_j\widehat{X}_{\overline{\jmath}}^2(t)
  - \sum_{k\in \O j}c_k\widehat{X}_j(t)\widehat{X}_k(t).
\end{align*}

We can now consider  the self similar solutions for the tree
dyadic model, as introduced in~\cite{BarBiaFlaMor12},
$X_j(t)=\dfrac{a_j}{t-t_0}$, defined for $t>t_0$, with $t_0<0$ and
with $(a_j)_{j\in J}\in l^2$ such that
\[
\begin{cases}
a_{\bar0}=0\\
a_j + c_j a_{\bar\jmath}^2 = \displaystyle{\sum_{k\in\O j}c_ka_ja_k},  \qquad\forall j\in J.
\end{cases}
\]

We time-reverse them
and we define
\[
  \widehat{X}_j(t)=-X_j(-t)\qquad \forall j\in J \qquad t< -t_0,
\]
which, as we pointed out earlier, is a solution
of~\eqref{tree-like_dyadic_det} in $(-\infty,-t_0)$, with $-t_0>0$. 
Since
\[
  \lim_{t\to +\infty}|X_j(t)|=0 \quad\textrm{and}\quad \lim_{t\to
    t_0^+}|X_j(t)|=+\infty,\ \forall j\in J 
\]
we have
\[
  \lim_{t\to -\infty}|\widehat{X}_j(t)|=0 \quad\textrm{and}\quad \lim_{t\to
    -t_0^-}|\widehat{X}_j(t)|=+\infty,\ \forall j\in J.
\]
Thanks to theorem~\ref{thm:exsoll2} there is a solution
$\widetilde{X}$, with initial conditions $x=\widehat{X}(0)$, and this
solution is a finite energy one, so, in particular, doesn't blow up in
$-t_0$. Yet it has the same initial conditions of $\widehat{X}$, so we
can conclude that there is no uniqueness of solutions in the
deterministic case.

\section{It\^o formulation}

Let's write the infinite dimensional
system~\eqref{tree-like_dyadic_sto} in It\^o formulation:

\begin{multline}
  \label{eq:dyad_tree_ito}
  dX_j=(c_{j}X_{\overline{\jmath}}^{2}-\sum_{k\in\O j}
c_{k}X_{j}X_{k}) dt +\sigma  c_j X_{\overline \jmath} dW_j \\
-
\sigma \sum_{k\in \O j}c_kX_k dW_k -
\frac{\sigma^2}{2}(c_j^2+\sum_{k\in \O j}c_k^2)X_jdt.
\end{multline}
We will use this formulation since it's easier to handle the
calculations, while all results can also be stated in the Stratonovich
formulation. 

So let's now introduce the definition of weak solution.
 A filtered probability space $(\Omega, \mathcal{F}_t,P)$ is a
 probability space $(\Omega, \mathcal{F}_\infty,P)$ together with a
 right-continuous filtration $(\mathcal{F}_t)_{t\geq 0}$ such that
 $\mathcal{F}_\infty$ is the $\sigma$-algebra generated by
 $\bigcup_{t\geq 0}\mathcal{F}_t$.

\begin{defi}\label{def:solution}
  Given $x\in l^2$, a weak solution of~\eqref{tree-like_dyadic_sto} in
  $l^2$ is a filtered probability space $(\Omega, \mathcal{F}_t,P)$, a
  $J$-indexed sequence of independent Brownian motions $(W_j)_{j\in
    J}$ on $(\Omega, \mathcal{F}_t,P)$ and an $l^2$-valued process
  $(X_j)_{j\in J}$ on $(\Omega, \mathcal{F}_t,P)$ with continuous
  adapted components $X_j$ such that
  \begin{multline}
\label{eq:solution}
    X_j= x_j + \int_0^t[c_jX_{\overline{\jmath}}^{2}(s)-\sum_{k\in\O j}
c_{k}X_{j}(s)X_{k}(s)]ds \\
+ \int_0^t\sigma c_j X_{\overline{\jmath}}(s)dW_j(s) -\sum_{k\in \O
  j}\int_0^t \sigma c_kX_k(s)dW_k(s)\\
-\frac{\sigma^2}{2}\int_0^t (c_j^2+\sum_{k\in \O j}c_k^2)X_j(s)ds,
  \end{multline}
for every $j\in J$, with $c_0=0$ and $X_0(t)=0$. We will denote this
solution by
\[
(\Omega, \mathcal{F}_t,P,W,X),
\]
or simply by $X$. 
\end{defi}

\begin{defi}
 A weak solution is an energy controlled solution if it is a solution
  as in Definition~\ref{def:solution} and it satisfies
\[
   P(\sum_{j\in J}X_j^2(t)\leq \sum_{j\in J}x_j^2)=1,
\]
for all $t\geq 0$.
\end{defi}

\begin{thm}
  \label{thm:existsweaksol}
  There exists an energy controlled solution
  to~\eqref{eq:dyad_tree_ito} in $L^\infty(\Omega\times[0,T],l^2)$ for $(x_j)\in l^2$.
\end{thm}

We will give a proof of this Theorem at the end of
Section~\ref{sec:uniqueness}. It is a weak existence result and uses
the Girsanov transform.

We'll prove in the following result that a process
satisfying~\eqref{eq:dyad_tree_ito}
satisfies~\eqref{tree-like_dyadic_sto} too.

\begin{prop}
  If $X$ is a weak solution, for every $j\in J$ the process
  $(X_j(t))_{t\geq 0}$ is a continuous semimartingale, so the 
  following equalities hold:
  \begin{align*}
    \int_0^t \sigma  c_j X_{\overline \jmath}(s)\circ
    dW_j(s)&=\int_0^t\sigma c_j X_{\overline{\jmath}}(s)dW_j(s) 
-\frac{\sigma^2}{2}\int_0^t c_j^2X_j(s)ds\\
    \int_0^t \sigma \sum_{k\in \O j}c_kX_k(s)\circ dW_k(s)&=\sum_{k\in \O
  j}\int_0^t \sigma c_kX_k(s)dW_k(s) + \frac{\sigma^2}{2}\int_0^t\sum_{k\in \O j}c_k^2X_j(s)ds,
  \end{align*}
where the Stratonovich integrals are well defined.
  So $X$ satisfies the Stratonovich formulation of the problem~\eqref{tree-like_dyadic_sto}.
\end{prop}

\begin{proof}
  We know that 
  \[
    \int_0^t \sigma  c_j X_{\overline \jmath}(s)\circ dW_j(s) =
    \int_0^t \sigma  c_j X_{\overline \jmath}(s) dW_j(s) +
    \frac{\sigma c_j}{2}[\xfat,W_j]_t,
  \]
  but from~\eqref{tree-like_dyadic_sto} we have that the only
  contribution to $[\xfat,W_j]$ is given by the $-\sigma c_j X_j\circ
  dW_j$ term, so 
  \[
    \frac{\sigma c_j}{2}[\xfat,W_j]_t=\frac{\sigma c_j}{2}[-\int_0^t \sigma c_j X_j\circ
  dW_j,W_j]_t=-\frac{\sigma^2 c_j^2}{2}\int_0^tX_jds.
  \]
  Now if we consider the other integral, we have
  \[
    \int_0^t \sigma \sum_{k\in \O j}c_kX_k(s)\circ dW_k(s)=\sum_{k\in \O
  j}\int_0^t \sigma c_kX_k(s)dW_k(s) + \sum_{k\in \O j}\frac{\sigma c_k}{2}[X_k,W_k]_t.
  \]
  For each $X_k$ we get, with the same computations,
  that the only contribution to $[X_k,W_k]_t$ comes from the term $\sigma c_k X_j\circ
  dW_k$, so that we get 
  \[
    \frac{\sigma c_k}{2}[X_k,W_k]_t=\frac{\sigma c_k}{2}[\int_0^t \sigma c_k X_j\circ
  dW_k,W_k]_t=\frac{\sigma^2 c_k^2}{2}\int_0^tX_jds.\qedhere
  \]
\end{proof}

\section{Girsanov transform}\label{sec:GirsanovTransform}
Let's consider~\eqref{eq:dyad_tree_ito} and rewrite it as
\begin{multline}\label{eq:inutile}
   dX_j=c_{j}X_{\overline{\jmath}}(X_{\overline{\jmath}}dt+\sigma
   dW_j)-\\\sum_{k\in\O j}
c_{k}X_{k}(X_{j} dt +\sigma  dW_k)
-\frac{\sigma^2}{2}(c_j^2+\sum_{k\in \O j}c_k^2)X_jdt.
\end{multline}
The idea is to isolate $X_{\overline{\jmath}}dt+\sigma
   dW_j$ and prove through Girsanov's theorem that they are Brownian
   motions with respect to a new measure $\widehat{P}$ in
   $(\Omega,\mathcal{F}_\infty)$, simultaneously for every $j\in J$.
   This way~\eqref{eq:dyad_tree_ito} becomes a system of linear SDEs
   under the new measure~$\widehat{P}$. The infinite dimensional
   version of Girsanov's theorem can be found in~\cite{1109.0363}
   and~\cite{MR2446690}. 

\begin{rmk}
We can obtain the same result under Stratonovich formulation.
\end{rmk}

Let $X$ be an energy controlled solution: its energy $\energy(t)$ is
bounded, so we can define the process
\begin{equation}\label{eq:defM}
M_t=-\frac{1}{\sigma}\sum_{j\in J}\int_0^tX_{\overline{\jmath}}(s)dW_j(s)
\end{equation}
which is a martingale. Its quadratic
variation is
\[
[M]_t= \frac{1}{\sigma^2}\int_0^t\sum_{j\in J}X_{\overline{\jmath}}^2(s)ds.
\]
Because of the same boundedness of $\energy(t)$ stated above, by the
Novikov criterion $\exp(M_t-\frac12 [M]_t)$ is a (strictly) positive
martingale. 
We now define $\widehat{P}$ on $(\Omega,\mathcal{F}_t)$ as
\begin{align}
   \frac{d\widehat{P}}{dP}\Big|_{\mathcal{F}_t} & = \exp(M_t-\frac12[M]_t)
   \nonumber\\
    \label{eq:def_widehatP} & = \exp(-\frac{1}{\sigma}\sum_{j\in
      J}\int_0^tX_{\overline{\jmath}}(s)dW_j(s) - \frac{1}{2\sigma^2}\int_0^t\sum_{j\in J}X_{\overline{\jmath}}^2(s)ds),
\end{align}
for every $t\geq 0$. $P$ and $\widehat{P}$ are equivalent on each
$\mathcal{F}_t$, because of the strict positivity of the exponential.

We can now prove the following:
\begin{thm}\label{thm:lin}
If $(\Omega, \mathcal{F}_t,P,W,X)$ is an energy controlled solution of
the nonlinear equation~\eqref{eq:dyad_tree_ito}, then $(\Omega,
\mathcal{F}_t,\widehat{P},B,X)$ satisfies the
linear equation 
\begin{equation}
   dX_j=\sigma c_jX_{\overline{\jmath}}dB_j(t)-\sigma \sum_{k\in \O
    j}c_kX_kdB_k(t)
 -\frac{\sigma^2}{2}(c_j^2+\sum_{k\in \O j}c_k^2)X_jdt,\label{linear}
\end{equation}  
where the processes
\[
B_j(t) = W_j(t) + \int_0^t\frac1{\sigma}X_{\overline{\jmath}}(s)ds
\]
are a sequence of independent Brownian motions on $(\Omega,
\mathcal{F}_t,\widehat{P})$, with $\widehat{P}$ defined by~\eqref{eq:def_widehatP}.
\end{thm}

\begin{proof}
Now let's define 
\[
B_j(t) = W_j(t) + \int_0^t\frac1{\sigma}X_{\overline{\jmath}}(s)ds.
\]

Under $\widehat{P}$, $(B_j(t))_{j\in J, t\in [0,T]}$ is a sequence of independent
Brownian motions.

Since
\begin{align*}
  \sigma \int_0^t c_j \xfat(s) dB_j(s) &= \sigma \int_0^t c_j \xfat(s)
  dW_j(s) + \int_0^t c_j \xfat^2(s)ds\\
\sigma \int_0^t c_kX_k(s) dB_k(s) &= \sigma\int_0^t c_kX_k(s)dW_k(s) +
\int_0^t c_kX_j(s)X_k(s)ds \quad k\in\O j.
\end{align*}
Then~\eqref{eq:inutile} can be rewritten in integral form as

\begin{multline}\label{eq:linear_integralform}
  X_j(t) = x_j + \sigma\int_0^tc_j\xfat (s) dB_j(s) - \sigma
  \sum_{k\in \O j}\int_0^t c_kX_k(s) dB_k(s)\\
 - \frac{\sigma^2}{2}\int_0^t (c_j^2+\sum_{k\in \O j}c_k^2)X_j(s)ds,
\end{multline}
which is a linear stochastic equation. 
\end{proof}

\begin{rmk}
  We can write our linear equation~\eqref{linear} also in Stratonovich
  form:
\[
dX_j = \sigma c_j \xfat \circ dB_j(t) - \sum_{k\in \O j} \sigma c_k X_k \circ dB_k(t).
\]
\end{rmk}

\begin{rmk}
  If we look at~\eqref{linear} we can see that it is possible to drop
  the $\sigma$, considering it a part of the coefficients $c_j$.
\end{rmk}

\todo[inline,caption={quadratic variation}]{In equation (10) the author writes $[X_j]_t$ while he
  uses another notation elsewhere. I suggest that he sticks with only
  one formulation for consistency.}
\todo[inline, color=green,nolist]{Changed notation in equation (7) and surroundings}
We can use It\^o formula to calculate
\begin{align}
  \label{eq:itoquadrato} \frac12 dX_j^2 & = X_jdX_j + \frac12
  d[X_j]_t\\
  &= \sigma c_j \xfat X_j dB_j - \sigma \sum_{k \in \O j}c_k X_j
  X_kdB_k\nonumber\\
& \quad -\frac{\sigma^2}{2}(c_j^2+\sum_{k\in \O
      j}c_k^2)X_j^2dt + \frac{\sigma^2}{2}(c_j^2X_{\overline{\jmath}}^2+\sum_{k\in \O j}c_k^2X_k^2)dt\nonumber\\
& = -\frac{\sigma^2}{2}(c_j^2+\sum_{k\in \O
      j}c_k^2)X_j^2dt + dN_j +
    \frac{\sigma^2}{2}(c_j^2X_{\overline{\jmath}}^2+\sum_{k\in \O
      j}c_k^2X_k^2)dt,\nonumber
\end{align}
with 
\begin{equation}\label{eq:def_N}
N_j(t)= \sigma\int_0^tc_jX_{\overline{\jmath}}X_jdB_j - \sigma\sum_{k\in \O j}\int_0^tc_kX_jX_kdB_k.
\end{equation}
This equality will be useful in the following.

We now present an existence result also for system~\eqref{linear}.

\begin{prop}\label{prop:existlinear}
  There exists a solution of~\eqref{linear} in
  $L^\infty(\Omega\times[0,T],l^2)$ with continuous
  components, with initial conditions $x\in l^2$.
\end{prop}

\begin{proof}
  Fix $N\geq1$ and consider the finite dimensional stochastic linear system
\begin{equation}\label{eq:galerkinsto}
\begin{cases}
 dX_j^N=\sigma c_jX_{\overline{\jmath}}^NdB_j(t)-\sigma \sum_{k\in \O
    j}c_kX_k^NdB_k(t)\\
 \qquad \quad -\frac{\sigma^2}{2}(c_j^2+\sum_{k\in \O j}c_k^2)X_j^Ndt
     & j\in J,\ 0\leq |j|\leq N\\
X_k^N(t)\equiv 0 & k\in J,\ |k|=N+1\\
X_j^N(0)=x_j &j\in J,\ 0\leq|j|\leq N.
\end{cases}
\end{equation}
This system has a unique global strong solution $(X_j^N)_{j\in J}$. We can compute,
using~\eqref{eq:itoquadrato} and the definition of $N_j$ in~\eqref{eq:def_N},
\begin{align*}
   \frac{1}{2} d (\sum_{|j|\leq N}(X_j^N(t))^2) & = \sum_{|j|\leq N}(-\frac{\sigma^2}{2}(c_j^2+\sum_{k\in \O
       j}c_k^2)(X_j^N)^2dt + dN_j^N \\
&\quad +
     \frac{\sigma^2}{2}(c_j^2(X_{\overline{\jmath}}^N)^2+\sum_{k\in \O
       j}c_k^2(X_k^N)^2)dt) \\
&=-\sum_{|j|=N}\frac{\sigma^2}{2}\sum_{k\in \O j}c_k^2(X_j^N)^2dt\\
&=-\frac{\sigma^2}{2}\sum_{|k|=N+1}c_k^2(X_{\overline k}^N)^2\leq 0.
\end{align*}
Hence
\[
  \sum_{|j|\leq N}(X_j^N(t))^2 \leq \sum_{|j|\leq N}x_j^2\leq
  \sum_{j\in J}x_j^2 \qquad
  \widehat{P}-\textrm{a.s.}\quad \forall t\geq 0.
\]
This implies that there exists a sequence $N_m\uparrow\infty$ such
that $(X_j^{N_m})_{j\in J}$ converges weakly to some $(X_j)_{j\in J}$ in
$L^2(\Omega\times[0,T],l^2)$ and also weakly star in
$L^\infty(\Omega\times[0,T],l^2)$, so $(X_j)_{j\in J}$ is in
$L^\infty(\Omega\times[0,T],l^2)$. 

Now for every $N \in \N$, $(X_j^N)_{j\in J}$ is in$\prog$, the
subspace of progressively measurable processes in
$L^2(\Omega\times[0,T],l^2)$. But $\prog$ is 
strongly closed, hence weakly closed, so $(X_j)_{j\in J}\in \prog$.

We just have to prove that $(X_j)_{j\in J}$ solves~\eqref{linear}. All the
one dimensional stochastic integrals that appear in each equation
in~\eqref{eq:linear_integralform} are linear strongly continuous
operators $\prog\to L^2(\Omega)$, hence weakly continuous. Then we can
pass to the weak limit in~\eqref{eq:galerkinsto}. Moreover from the
integral equations~\eqref{eq:linear_integralform} we have that there
is a modification of the solution which is continuous in all the components.
\end{proof}

\section{Closed equation for $E_{\widehat{P}}[X_j^2(t)]$}

\begin{prop}
For every energy controlled solution $X$ of the nonlinear
equation~\eqref{eq:dyad_tree_ito}, $E_{\widehat{P}}[X_j^2(t)]$ is finite for every
$j\in J$ and satisfies
\begin{multline}
  \label{eq:secondmoment}
  \frac{d}{dt}E_{\widehat{P}}[X_j^2(t)]= -\sigma^2(c_j^2+\sum_{k\in \O
      j}c_k^2)E_{\widehat{P}}[X_j^2(t)]\\+
    \sigma^2c_j^2E_{\widehat{P}}[X_{\overline{\jmath}}^2(t)]+\sigma^2\sum_{k\in \O
      j}c_k^2E_{\widehat{P}}[X_k^2(t)].
\end{multline}
\end{prop}

\begin{proof}
Let $(\Omega, \mathcal{F}_t,P,W,X)$ be an energy controlled solution of
the nonlinear equation~\eqref{eq:dyad_tree_ito}, with initial
condition $X\in l^2$ and let $\widehat{P}$ be the measure given by
Theorem~\ref{thm:lin}. Denote by $E_{\widehat{P}}$ the expectation with respect to
$\widehat{P}$ in $(\Omega, \mathcal{F}_t)$. 


Notice that
\begin{equation}
  \label{eq:boundfourth}
  E_{\widehat{P}}[\int_0^TX_j^4(t)dt]<\infty \qquad \forall j\in J.
\end{equation}

For energy controlled solutions from the definition we have that
$P$-a.s.
\[
\sum_{j\in J}X_j^4(t) \leq \max_{j\in J}X_j^2(t) \sum_{j\in
  J}X_j^2(t)\leq (\sum_{j\in J}x_j^2)^2,
\]
because of the behavior of the energy we showed. But on every
$\mathcal{F}_t$, $P\sim \widehat{P}$, so
\[
\widehat{P}(\sum_{j\in J}X_j^4(t)\leq (\sum_{j\in J}x_j^2)^2)=1,
\]
and~\eqref{eq:boundfourth} holds.

From~\eqref{eq:boundfourth} it follows that $M_j(t)$ is a martingale
for every $j\in J$. Moreover
\[
E_{\widehat{P}}[\sum_{j\in J}X_j^2(t)]<\infty,
\]
since $X_j(t)$ is an energy controlled solution and the condition is
invariant under the change of measure $P\leftrightarrow \widehat{P}$ on
$\mathcal{F}_t$ and, in particular,

\[
E_{\widehat{P}}[X_j^2(t)]<\infty\qquad \forall j\in J.
\]

Now let's write~\eqref{eq:itoquadrato} in integral form:
\begin{multline*}
X_j^2(t)-x_j^2 = -\sigma^2\int_0^t(c_j^2+\sum_{k\in \O
      j}c_k^2)X_j^2(s)ds \\+ 2\int_0^tdN_j(s) +
    \sigma^2\int_0^t(c_j^2X_{\overline{\jmath}}^2(s)+\sum_{k\in \O
      j}c_k^2X_k^2(s))ds.
\end{multline*}
We can take the $\widehat{P}$ expectation,
\begin{multline*}
E_{\widehat{P}}[X_j^2(t)]-x_j^2 = -\sigma^2\int_0^t(c_j^2+\sum_{k\in \O
      j}c_k^2)E_{\widehat{P}}[X_j^2(s)]ds \\+
    \sigma^2\int_0^tc_j^2E_{\widehat{P}}[X_{\overline{\jmath}}^2(s)]ds+\sigma^2\sum_{k\in \O
      j}\int_0^tc_k^2E_{\widehat{P}}[X_k^2(s)]ds,
\end{multline*}
where the $N_j$ term vanishes, since it's a $\widehat{P}$-martingale. Now we can
derive and the proposition is established.
\end{proof}

  It's worth stressing that $E_{\widehat{P}}[X_j^2(t)]$ satisfies a closed
  equation. Even more interesting is the fact that this is the forward
  equation of a continuous-time Markov chain, as we will see in the
  following section.

\section{Associated Markov chain}
\label{sec:assmarkcha}

  We want to show and use this characterization of the second moments
  equation as the forward equation of a Markov chain, taking advantage of some known
  results in the Markov chains theory. We follow the transition
  functions approach to continuous times Markov chains; we don't
  assume any knowledge of this theory, so we will provide the basic
  definitions and results we need. More results can be found in the
  literature, see for example~\cite{MR1118840}.

  \begin{defi}
    A non-negative function $f_{j,l}(t)$ with $j,l \in J$ and $t\geq 0$ is a
    \emph{transition function} on $J$ if $f_{jl}(0)=\delta_{jl}$,
    \[
      \sum_{l\in J}f_{jl}(t)\leq 1 \qquad \forall j\in J, \ \forall
      t\geq 0,
    \]
    and it satisfies the semigroup property (or Chapman-Kolmogorov
    equation)
    \[
      f_{jl}(t+s)=\sum_{h\in J}f_{jh}(t)f_{hl}(s)\qquad \forall j,l\in
      J,\ \forall t,s\geq 0.
    \]
  \end{defi}

  \begin{defi}
    A \emph{$q$-matrix} $Q=(q_{jl})_{j,l\in J}$ is a square matrix
    such that
    \begin{gather*}
      0\leq q_{jl}<+\infty\qquad \forall j\neq l \in J,\\
      \sum_{l\neq j}q_{jl}\leq -q_{jj}\eqqcolon q_j \leq +\infty\qquad
      \forall j\in J.
    \end{gather*}
    A $q$-matrix is called \emph{stable} if all $q_j$'s are finite,
    and \emph{conservative} if
    \[
      q_j=\sum_{l\neq j}q_{jl}\qquad \forall j\in J.
    \]

    If $Q$ is a $q$-matrix, a $Q$-function is a transition function
    $f_{jl}(t)$ such that $f^\prime_{jl}(0)=Q$.
  \end{defi}

  The $q$-matrix shows a close resemblance to the infinitesimal
  generator of the transition function, but they differ, since the
  former doesn't determine a unique transition function, while the
  latter does. Still this approach can be seen as a generator approach
  to Markov chains in continuous times.

  Now let's see these objects in our framework: let's
  write~\eqref{eq:secondmoment} in matrix form. Let $Q$ be the
  infinite dimensional matrix which entries are defined as 
\[ 
    q_{j,j} = -\sigma^2(c_j^2+\sum_{k\in \O
      j}c_k^2) \quad q_{j,\bar\jmath} = \sigma^2c_j^2 \quad
    q_{j,k}=\mathbbm{1}_{\{k\in \O j\}}\sigma^2c_k^2 \ \textrm{for
    }k\neq j,\bar\jmath
\]

\begin{prop}\label{prop:qmatrix}
  The infinite matrix $Q$ defined above is the stable and conservative
  $q$-matrix. Moreover $Q$ is symmetric.
\end{prop}

\begin{proof}
  It's easy to check that $Q$ is a stable and conservative
  $q$-matrix. First of all $q_{j,j}<0$ for all $j\in J$ and
  $q_{j,l}\geq 0$ for all $j\neq l$. Then
\[
q_{j}=\sum_{l\neq j}q_{j,l}=q_{j,\overline \jmath}+\sum_{k\in \O
  j}q_{j,k} = \sigma^2c_j^2 + \sum_{k\in \O
  j}\sigma^2c_k^2 = -q_{j,j}.
\]

Moreover it is very easy to check that the matrix is symmetric:
\[
  q_{ij}=\left\{
    \begin{aligned}
      \sigma^2 c_j^2 \quad 
      l=\overline{\jmath}\quad & \Leftrightarrow
      \quad j\in \O l\quad \sigma^2 c_j^2\\
      \sigma^2 c_l^2\quad l\in \O j\quad &\Leftrightarrow
      \quad j=\overline{l}\quad \sigma^2 c_l^2
    \end{aligned}
\right\}=q_{lj}
\]
\end{proof}

Since $Q$ is a $q$-matrix we can construct the process associated,  as a jump
and hold process on the space state, which in our case is the tree of
the dyadic model. The process will wait in node $j$ for an exponential 
time of parameter $q_j$, and then will jump to $\bar\jmath$ or
$k\in \O j$ with probabilities $q_{j,\bar\jmath}/q_j$ and
$q_{j,k}/q_j$ respectively. This process is a continuous time Markov
chain that has $J$ as a state space and also has the same skeleton as
the dyadic tree model, meaning that the transition probabilities are
non-zero only if one of the nodes is the father of the other one.

 Given a $q$-matrix $Q$, it is naturally associated with two
 (systems of) differential equations:
 \begin{gather}
  \label{eq:fwdkolm}
  y^\prime_{jl}(t) = \sum_{h\in J}y_{jh}(t) q_{hl}\\
  \nonumber y_{jl}^\prime(t) = \sum_{h\in J}q_{jh} y_{hl}(t),
 \end{gather}
 called forward and backwards Kolmogorov equations, respectively.

\begin{lem}\label{lemma:uniq_fwd}
  Given a stable, symmetric and conservative $q$-matrix $Q$, 
  then the unique nonnegative solution of the
  forward equations~\eqref{eq:fwdkolm} in
  $L^\infty([0,\infty),l^1)$, given a null
  initial condition $y(0)=0$, is $y(t)=0$.
\end{lem}

\begin{proof}
  Let $y$ be a generic solution, then
  \begin{equation}
    \label{eq:anderson2.8}
    \left\{
      \begin{aligned}
        &\frac{d}{dt}y_j(t)=\sum_{i\in J}y_i(t)q_{ij}\\
        &y_j(t)\geq 0 \quad j\in J\\
        &y_j(0)=0 \quad j\in J\\
        &\sum_{j\in J}y_j(t) < +\infty.
      \end{aligned}
    \right.
  \end{equation}

We can consider for every node
$\hat{y}_j=\int_0^{+\infty}e^{-t}y_j(t)dt$, the Laplace transform in
1. From the last equation of the system above, we have $\sum_j
\hat{y}_j\leq M$, for some constant $M>0$, so in particular we can
consider $k \in J$ such that $\hat{y}_k\geq \hat{y}_j$, for all $j\in
J$.

Now we want to show that $y_k^\prime (t)$ is bounded: thanks to
the symmetry and stability of $Q$ we have
\[
  |y_k^\prime (t)|\leq |-q_ky_k(t)| + |\sum_{l\neq k}y_l(t)q_{lk}|
  \leq q_kM + q_kM < +\infty.
\]
We can integrate by parts
\begin{multline}
  \hat{y}_k = \int_0^{+\infty}e^{-t}y_k^\prime(t)dt =
  \int_0^{+\infty}e^{-t}\sum_{l\in J}y_l(t)q_{lk}dt = \sum_{l\in
    J}\hat{y}_lq_{lk}\\
  = -\hat{y}_kq_k + \sum_{l\neq k}\hat{y}_lq_{lk} \leq \hat{y}_k(-q_k+\sum_{l\neq k}q_{kl})=0,
\end{multline}
where the last equality follows from the conservativeness of $Q$, and we
used the stability and symmetry. Now we have $\hat{y}_k=0$ and so all
$\hat{y}_j=0$, hence $y_j(t)=0$ for all $j\in J$, for all $t\geq 0$.
\end{proof}

\section{Uniqueness}\label{sec:uniqueness}
Now we can use the results of the previous section to prove the main
results of this paper.

\begin{thm}\label{thm:struniq}
    There is strong uniqueness for the linear system~\eqref{linear} in
    the class of energy controlled $L^\infty(\Omega \times [0,T],l^2)$ solutions.
  \end{thm}

  \begin{proof}
    By linearity of~\eqref{linear} it is enough to prove that for null
    initial conditions there is no nontrivial solution. Since we
    have~\eqref{eq:secondmoment}, proposition~\ref{prop:qmatrix}
    and lemma~\ref{lemma:uniq_fwd}, then $E_{\widehat{P}}[X_j^2(t)]=0$ for all $j$ and
    $t$, hence $X=0$ a.s.
  \end{proof}

  Let's recall that we already proved an existence result
  for~\eqref{linear} with proposition~\ref{prop:existlinear}. 

\begin{thm}\label{thm:weakuniq}
    There is  uniqueness in law for the nonlinear system~\eqref{eq:dyad_tree_ito} in
    the class of energy controlled $L^\infty(\Omega \times [0,T],l^2)$
    solutions. 
  \end{thm}

  \begin{proof}
  Assume that $(\Omega^{(i)},
  \mathcal{F}_t^{(i)},P^{(i)},W^{(i)},X^{(i)})$, $i=1,2$, are two
  solutions of~\eqref{eq:dyad_tree_ito} with the same initial
  conditions $x\in l^2$. Given $n\in \N$, $t_1,\ldots,t_n \in [0,T]$
  and a measurable and bounded function $f:(l^2)^n\to \R$, we want to
  prove that
\begin{equation}\label{eq:thesisuniqueness}
E_{P^{(1)}}[f(X^{(1)}(t_1),\ldots,X^{(1)}(t_n))]=E_{P^{(2)}}[f(X^{(2)}(t_1),\ldots,X^{(2)}(t_n))].
\end{equation}
By theorem~\ref{thm:lin} and the definition of $\widehat{P}$ given
in~\eqref{eq:def_widehatP} we have that, for $i=1,2$,
\begin{multline}
  \label{eq:EPi}
  E_{P^{(i)}}[f(X^{(i)}(t_1),\ldots,X^{(i)}(t_n))] =\\ E_{\widehat{P}^{(i)}}[\exp\{-M_T^{(i)}+\frac12[M^{i},M^{(i)}]_T\}f(X^{(i)}(t_1),\ldots,X^{(i)}(t_n))],
\end{multline}
where $M^{(i)}$is defined as in~\eqref{eq:defM}. We have proven in
proposition~\ref{prop:existlinear} and
theorem~\ref{thm:struniq} that the linear system~\eqref{linear} has a
unique strong solution. Thus it has uniqueness in law on
$\mathcal{C}([0,T], \R)^\N$ by
Yamada-Watanabe theorem, 
that is under the measures $\widehat{P}^{(i)}$, the processes
$X^{(i)}$ have the same laws. 
For a detailed proof of this theorem in infinite dimension see~\cite{1109.0363}.

Now we can also include $M^{(i)}$ in the system and conclude that
$(X^{(i)},M^{(i)})$ under $\widehat{P}^{(i)}$ have laws independent of $i=1,2$,
hence, through~\eqref{eq:EPi}, we have~\eqref{eq:thesisuniqueness}.
  \end{proof}

We can now conclude with the proof of Theorem~\ref{thm:existsweaksol}.

\begin{proof}[Proof of Theorem~\ref{thm:existsweaksol}]
  Let $(\Omega,\mathcal{F}_t,\widehat{P},B,X )$ be the solution of~\eqref{linear} in
  $L^\infty(\Omega \times [0,T],l^2)$ provided by
  theorem~\ref{thm:struniq}.
  We follow the same argument as in Section~\ref{sec:GirsanovTransform},
    only from $\widehat{P}$ to $P$. We construct $P$ as a measure on
    $(\Omega,\mathcal{F}_T)$ satisfying
    \[
      \frac{dP}{d\widehat{P}}\Big|_{\mathcal{F}_T}  = \exp(\widehat{M}_T-\frac12[\widehat{M},\widehat{M}]_T),
    \]
    where $\widehat{M}_t=\frac{1}{\sigma}\sum_{j\in
      J}\int_0^tX_{\overline{\jmath}}(s)dB_j(s)$.
    Under $P$ the processes 
    \[
      W_j(t)= B_j(t) - \int_0^t\frac1{\sigma}X_{\overline{\jmath}}(s)ds,
    \]
    are a sequence of independent Brownian motions. Hence $(\Omega,
    \mathcal{F}_t,P,W,X)$ is a solution of~\eqref{eq:dyad_tree_ito}
    and it is in $L^\infty$, since $P$ and $\widehat{P}$ are equivalent on
    $\mathcal{F}_T$.
\end{proof}





\bibliographystyle{abbrv}
\bibliography{bibdyadic}
\end{document}